\documentclass[12pt,reqno]{amsart}
\usepackage{amsmath,amssymb,amsthm,mathtools}
\usepackage[margin=1in]{geometry}
\usepackage{microtype,booktabs,array}
\usepackage{xcolor}
\usepackage[numbers,sort&compress]{natbib}
\usepackage[colorlinks=true,linkcolor=blue!50!black,
citecolor=blue!50!black,urlcolor=blue!50!black]{hyperref}
\hypersetup{pdftitle={An improved bound on the support diameter of nonnegative planar Euler vorticity}}
\allowdisplaybreaks
\numberwithin{equation}{section}

\newtheorem{theorem}{Theorem}[section]
\newtheorem{corollary}[theorem]{Corollary}
\newtheorem{proposition}[theorem]{Proposition}
\newtheorem{lemma}[theorem]{Lemma}

\theoremstyle{definition}

\newtheorem{remark}[theorem]{Remark}

\begin{document}

\title[On the support diameter of planar Euler vorticity]
{An improved bound on the support diameter of nonnegative planar Euler vorticity}

\author{Daomin Cao}
\address{State Key Laboratory of Mathematical Sciences, Academy of Mathematics and Systems Science, Chinese Academy of Sciences, Beijing 100190, P.R. China; and University of Chinese Academy of Sciences, Beijing 100049, P.R. China}
\email{dmcao@amt.ac.cn}

\author{Junhong Fan}
\address{Institute of Applied Mathematics, AMSS, Chinese Academy of Sciences, Beijing 100190; and University of Chinese Academy of Sciences, Beijing 100049, P.R. China}
\email{fanjunhong@amss.ac.cn}

\author{Guodong Wang}
\address{School of Mathematical Sciences, Dalian University of Technology, Dalian 116024, P.R. China}
\email{gdw@dlut.edu.cn}

\subjclass[2020]{Primary 35Q31; Secondary 76B47, 35B40.}
\keywords{Two-dimensional Euler equations; high-order moments;
support growth.}

\begin{abstract}
We prove an $O(t^{1/4})$ bound on the support radius of
nonnegative planar Euler vorticity with compactly supported
$L^1$ initial data, removing the logarithmic factors from the
classical confinement estimates. The result holds in the
symmetrized vorticity formulation. A factorization of the
interaction kernel yields a quadratic convolution inequality
for high-order moments. Retaining this convolution allows an
elementary comparison for finite sums of moments to control
the full support. We also prove that the squared support radius
is H\"older continuous in time with exponent $1/2$.
\end{abstract}

\maketitle

\section{Introduction}\label{sec:introduction}

The transport of vorticity is a basic feature of two-dimensional
incompressible Euler flows. For classical solutions, vorticity is
preserved along particle trajectories, but its support may deform.
Conservation of the area of a vortex patch does not by itself control
its diameter.
Quantitative confinement estimates describe how far vorticity can
travel from its initial support. For vorticity of one sign, the
conservation of the center of vorticity and the moment of inertia
provides additional information about this motion. The purpose of
this paper is to improve the resulting upper bound for the growth
of the support.

\subsection{The problem and earlier results}

Consider the two-dimensional incompressible Euler equations in
vorticity form,
\begin{equation}\label{eq:euler}
 \begin{cases}
 \partial_t\omega+\mathbf u\cdot\nabla\omega=0,
     &t>0,\quad\mathbf x\in\mathbb R^2,\\[2pt]
 \displaystyle
 \mathbf u(t,\mathbf x)=\frac1{2\pi}
 \int_{\mathbb R^2}\frac{(\mathbf x-\mathbf y)^\perp}
 {|\mathbf x-\mathbf y|^2}\omega(t,\mathbf y)\,\mathrm d\mathbf y,\\[6pt]
 \omega(0,\mathbf x)=\omega_0(\mathbf x),
 \end{cases}
\end{equation}
where $(x_1,x_2)^\perp=(-x_2,x_1)$. We assume that
\begin{equation}\label{eq:data}
 0\leq\omega_0\in L^1(\mathbb R^2),\qquad
 \operatorname{supp}\omega_0\text{ is compact},\qquad
 \omega_0\not\equiv0.
\end{equation}
The support of a measurable function means its essential support,
namely the complement of the largest open set on which it vanishes
almost everywhere.

We consider nonnegative solutions in the class
\begin{equation}\label{eq:solution-class}
 \omega\in C_{\mathrm w}([0,\infty);L^1(\mathbb R^2)),
\end{equation}
where $C_{\mathrm w}$ denotes continuity in the weak topology of
$L^1$. At low regularity, we use the symmetrized vorticity
formulation of \cite{BBC2016}.
For $\varphi\in C_c^\infty(\mathbb R^2)$, set
\[
 \mathcal H_\varphi(\mathbf x,\mathbf y)
 =\frac{[\nabla\varphi(\mathbf x)-\nabla\varphi(\mathbf y)]
                 \cdot(\mathbf x-\mathbf y)^\perp}
        {4\pi|\mathbf x-\mathbf y|^2},
 \qquad \mathbf x\neq\mathbf y,
\]
and assign the value zero on the diagonal.  We require $\omega$ to satisfy
\begin{equation}\label{eq:weak-formulation}
 \begin{aligned}
 &\int_{\mathbb R^2}\varphi(\mathbf x)\omega(t,\mathbf x)\,\mathrm d\mathbf x
 -\int_{\mathbb R^2}\varphi(\mathbf x)\omega_0(\mathbf x)\,\mathrm d\mathbf x\\
 &\qquad=\int_0^t\iint_{\mathbb R^2\times\mathbb R^2}
 \mathcal H_\varphi(\mathbf x,\mathbf y)
 \omega(s,\mathbf x)\omega(s,\mathbf y)
 \,\mathrm d\mathbf x\,\mathrm d\mathbf y\,\mathrm ds,
 \qquad t\geq0.
 \end{aligned}
\end{equation}
The kernel is bounded by $(4\pi)^{-1}\|D^2\varphi\|_{L^\infty}$,
where the norm of the Hessian is its operator norm. Thus the
right-hand side is well defined for $L^1$ vorticity.
Under \eqref{eq:solution-class}, integration by parts in time
identifies \eqref{eq:weak-formulation} with the space-time test
formulation in \cite{BBC2016}.

For each data satisfying \eqref{eq:data}, nonnegative global
solutions in this class exist. In particular, smooth
mollifications of the initial vorticity yield, along a subsequence,
solutions converging in $C([0,T];L^1)$ on every finite time
interval. The limits are transported by measure-preserving
Lagrangian flows; see \cite{BBC2016}.
If the initial vorticity also belongs to $L^p$, $1<p\leq\infty$,
these solutions preserve its $L^p$ norm.
Our estimate applies to every nonnegative solution satisfying
\eqref{eq:solution-class}--\eqref{eq:weak-formulation}.

If, in addition, $\omega\in L^\infty_{\mathrm{loc}}([0,\infty);L^p)$
for some $1<p\leq\infty$, then the Biot--Savart velocity belongs to
$L^\infty_{\mathrm{loc}}([0,\infty);L^2_{\mathrm{loc}}(\mathbb R^2))$.
In this case, \eqref{eq:weak-formulation} is equivalent to the usual
weak velocity formulation of Euler by
\cite{BBC2016}.
At $L^1$ regularity, the velocity need not have locally finite
kinetic energy, so we retain the symmetrized formulation.
For bounded initial vorticity, our class includes the unique Yudovich solution
\cite{Yudovich1963,MarchioroPulvirenti1994,MB2002}.

Define the initial circulation, center, and moment of inertia by
\begin{equation}\label{eq:invariants}
 \begin{aligned}
 \Gamma&=\int_{\mathbb R^2}\omega_0(\mathbf x)\,\mathrm d\mathbf x,
 \qquad\mathbf c=\frac1\Gamma\int_{\mathbb R^2}\mathbf x\omega_0(\mathbf x)
                       \,\mathrm d\mathbf x,\qquad
 J=\int_{\mathbb R^2}|\mathbf x-\mathbf c|^2\omega_0(\mathbf x)
                       \,\mathrm d\mathbf x.
 \end{aligned}
\end{equation}
Here $\Gamma,J>0$: if $J=0$, the nonnegative integrable function
$\omega_0$ would vanish almost everywhere outside the single point
$\mathbf c$, contradicting $\Gamma>0$.
Lemma~\ref{lem:conservation} proves
conservation of these quantities directly from
\eqref{eq:weak-formulation}. Let $B_r(\mathbf a)$ denote the open
disk of radius $r$ centered at $\mathbf a$. Set
\begin{equation}\label{eq:radius}
 R(t)=\sup\{|\mathbf x-\mathbf c|:
                  \mathbf x\in\operatorname{supp}\omega(t,\cdot)\}.
\end{equation}
The value $+\infty$ is allowed in this definition until confinement
has been proved. We write $R_0=R(0)>0$, so that
$\operatorname{supp}\omega_0\subset\overline B_{R_0}(\mathbf c)$.
All asymptotic bounds below concern $t\to\infty$ with the initial
data fixed.

For bounded vorticity, the uniform velocity bound gives the
preliminary estimate $R(t)=O(t)$; see
\cite[Section~2]{ISG1999}. Nonnegativity and conservation of $J$
also yield
\begin{equation}\label{eq:inertia-tail}
 \int_{\mathbb R^2\setminus\overline B_r(\mathbf c)}\omega(t,\mathbf x)
       \,\mathrm d\mathbf x\leq\frac{J}{r^2},
 \qquad r>0.
\end{equation}
This controls the vorticity mass at large distances. To bound the
support itself, it is also necessary to control the motion of portions of
vorticity carrying arbitrarily small mass. The classical arguments
combine estimates such as \eqref{eq:inertia-tail} with the cancellation
of the first moment in the radial component of the velocity.

Marchioro \cite{Marchioro1994} proved that $R(t)=O(t^{1/3})$ for
bounded vorticity of one sign. Lopes Filho and Nussenzveig Lopes
\cite{Lopes1998} extended this estimate to compactly supported
vorticities of one sign in $L^p$, with $p>2$. Iftimie, Sideris, and
Gamblin \cite{ISG1999} subsequently obtained
\begin{equation}\label{eq:classical-bound}
 R(t)=O\bigl((t\log t)^{1/4}\bigr).
\end{equation}
The same estimate holds for compactly supported nonnegative
initial vorticity in $L^p$, $2<p\leq\infty$, as noted in
the concluding remark of \cite{ISG1999}.
Their proof estimates the vorticity outside large disks and then
controls the radial velocity on a suitable expanding boundary.
The appendix of \cite{ISG1999} contains an alternative proof, due
to Gamblin, based on high-order moments.

Serfati independently obtained bounds with any fixed number of
iterated logarithms. In the form recorded by Iftimie
\cite[Remark~2]{Iftimie1999}, these bounds read
\begin{equation}\label{eq:serfati-bound}
 R(t)\leq\widetilde C_k\bigl[t\log^{\circ k}t\bigr]^{1/4},
 \qquad t\geq t_k
\end{equation}
for every fixed integer $k\geq1$. Here $\log^{\circ k}$ denotes the
$k$-fold composition of the natural logarithm, and $\widetilde C_k,t_k$ may
depend on $k$ and the initial data. The result is attributed in
\cite{Iftimie1999} to the unpublished manuscript
\cite{SerfatiManuscript}. The dependence of
the constants on $k$ means that \eqref{eq:serfati-bound} alone does
not imply a bound of order $t^{1/4}$.

The sign condition is essential. Iftimie, Sideris, and Gamblin
\cite{ISG1999} proved linear growth of the support
diameter for initial data obtained by extending a nonzero bounded
nonnegative blob, compactly supported in the first quadrant, oddly
across each coordinate axis. Using the orbital stability of the
Lamb dipole established by Abe and Choi \cite{AbeChoi2022}, Choi
and Jeong \cite{ChoiJeong2022} constructed smooth,
compactly supported vorticity with changing sign whose support
diameter grows at least linearly for all sufficiently large times. Choi, Jeong, and Yao \cite{CJY2024} proved 
orbital stability of two sufficiently separated Lamb dipoles
moving away from each other, under perturbations that are odd
with respect to both coordinate axes and nonnegative in the
first quadrant. Abe, Jeong, and Yao \cite{AJY2025} established
stability of finite sums of Lamb dipoles for nonnegative
vorticity in the upper half-plane, assuming sufficient initial
separation and distinct propagation speeds ordered increasingly
from left to right.
Choi, Jeong, and Sim \cite{CJS2025} proved the existence of
Sadovskii vortex patches. Abe, Choi, Jeong, Sim, and Woo
\cite{ACJSW2026} described sign-changing perturbations
of such dipoles with a thin trailing region whose support diameter
grows at least linearly in time.

For compactly supported $L^p$ initial vorticity with
$1\leq p\leq2$, Hounie, Lopes Filho, and Nussenzveig Lopes
\cite{HLL1999} obtained uniform estimates on the vorticity mass
and the measure of its support outside large disks along smooth
approximations, together with corresponding a priori estimates
for weak limits.

For bounded, nonnegative and compactly supported vorticity outside
a smooth bounded simply connected obstacle, Marchioro
\cite{Marchioro1996} proved $O(t^{1/2+\varepsilon})$ confinement
for every $\varepsilon>0$ and $O(t^{1/3})$ confinement outside a
disk. Iftimie, Lopes Filho, and Nussenzveig Lopes \cite{ILL2007}
improved these bounds to $O(t^{1/2})$ and, for vorticity even about
the center of a disk, $O((t\log t)^{1/4})$, respectively.
These exterior flows include a harmonic velocity determined by
the circulation around the obstacle.

Ambrose, Lopes Filho, and Nussenzveig Lopes
\cite{Ambrose2018} obtained $O((t\log t)^{1/4})$ confinement of
unfiltered Euler-$\alpha$ vorticity with nonnegative compactly
supported bounded Radon measures as initial data.
For bounded, nonnegative and compactly supported vorticity in an
infinite cylinder, Choi and Denisov \cite{ChoiDenisov2019} bounded
the support diameter in the unbounded direction by
$O(t^{1/3}\log^2t)$.
Butt\`a and Cavallaro \cite{ButtaCavallaro2026}
refined this estimate to $o((t(\log t)^\beta)^{1/3})$ for every
fixed $\beta>1$. In the half-plane, Chen and Sun \cite{ChenSun2026} established the linear growth of the support diameter for smooth nonnegative initial data approximating any nonzero bounded, compactly supported, nonnegative vorticity in $L^1$ sense.
The odd extension of such vorticity to the whole plane changes sign.

Marchioro and Pulvirenti \cite{MarchioroPulvirenti1993} studied
localization near point-vortex motion as the initial core size
tends to zero. Iftimie and Marchioro \cite{IM2018} treated a positive
blob transported by its own velocity and the prescribed field of
a self-similar expanding point-vortex configuration.
For any fixed sufficiently small $\varepsilon>0$, Zbarsky
\cite{Zbarsky2021} constructed full Euler solutions with three
vorticity components whose circulations include both signs.
Their separation is of order $t^{1/2}$, and their radii are
$O(t^{1/4+\varepsilon})$ about their moving centers.
These component estimates concern a different
question from confinement of the entire nonnegative support in
\eqref{eq:data}. See \cite{Iftimie2007} for a broader account of
large-time behavior in incompressible Euler flows.

\subsection{Main result}

Our main result removes the logarithmic factors from the preceding
whole-plane estimates.

\begin{theorem}\label{thm:main}
Let $\omega\in C_{\mathrm w}([0,\infty);L^1(\mathbb R^2))$ be
nonnegative and satisfy \eqref{eq:weak-formulation}.
Assume that $\omega_0\not\equiv0$ has compact support.
Let $J$ be defined by
\eqref{eq:invariants}, and set $R_0=R(0)$, where $R(t)$ is defined
in \eqref{eq:radius}. Then we have
\begin{equation}\label{eq:main}
 R(t)^2\leq R_0^2+\frac{16}{\sqrt{27\pi}}\sqrt{Jt},
 \qquad t\geq0.
\end{equation}
Consequently,
\begin{equation}\label{eq:asymptotic}
 R(t)=O(t^{1/4}),\qquad
 \operatorname{diam}(\operatorname{supp}\omega(t,\cdot))=O(t^{1/4}).
\end{equation}
\end{theorem}

\begin{remark}
Theorem~\ref{thm:main} also applies to nonzero, nonnegative
and compactly supported initial data in $L^p(\mathbb R^2)$,
$1<p\leq\infty$, since such data belong to $L^1(\mathbb R^2)$.
\end{remark}

Moreover, we have the following corollary.

\begin{corollary}\label{cor:time-holder}
Under the assumptions of Theorem~\ref{thm:main}, we have
\begin{equation}\label{eq:time-holder}
 |R(t)^2-R(s)^2|
 \leq\frac{16}{\sqrt{27\pi}}\sqrt{J|t-s|},
 \qquad s,t\geq0.
\end{equation}
\end{corollary}

Table~\ref{tab:bounds} compares whole-plane results for nonnegative,
compactly supported initial vorticity. The last row concerns
solutions in \eqref{eq:solution-class}--\eqref{eq:weak-formulation}.

\begin{table}[htbp]
\centering
\caption{Confinement estimates in the whole plane.}
\label{tab:bounds}
\small
\begin{tabular}{@{}>{\raggedright\arraybackslash}p{0.40\textwidth}
                  >{\raggedright\arraybackslash}p{0.18\textwidth}
                  >{\raggedright\arraybackslash}p{0.35\textwidth}@{}}
\toprule
Reference & Initial regularity & Estimate \\
\midrule
Marchioro \cite{Marchioro1994}
 & $L^\infty$ & $R(t)=O(t^{1/3})$ \\
\addlinespace[4pt]
Lopes Filho--Nussenzveig Lopes \cite{Lopes1998}
 & $L^p$, $2<p\leq\infty$ & $R(t)=O(t^{1/3})$ \\
\addlinespace[4pt]
Iftimie--Sideris--Gamblin \cite{ISG1999}
 & $L^p$, $2<p\leq\infty$ & $R(t)=O((t\log t)^{1/4})$ \\
\addlinespace[4pt]
Serfati, as reported in \cite[Remark~2]{Iftimie1999}
 & $L^\infty$ & $R(t)=O((t\log^{\circ k}t)^{1/4})$ \\
\addlinespace[4pt]
Theorem~\ref{thm:main}
 & $L^1$ & $R(t)=O(t^{1/4})$ \\
\bottomrule
\end{tabular}
\end{table}

In the Serfati row, $k\geq1$ is fixed before $t\to\infty$.

\subsection{Main steps of the proof}

The proof uses the high-order moment approach developed in the
appendix of \cite{ISG1999}. To explain the dependence on the moment
order, first translate the center of vorticity to the origin.
For bounded vorticity, Iftimie, Sideris, and Gamblin consider the moments
\[
 m_n(t)=\int_{\mathbb R^2}|\mathbf x|^{4n}
                  \omega(t,\mathbf x)\,\mathrm d\mathbf x,
 \qquad n\geq0
\]
and prove the estimates
\begin{equation}\label{eq:classical-moments}
 m_n'(t)\leq C Jn^2m_{n-1}(t),\qquad
 m_n(t)\leq\Gamma(R_0^4+C Jnt)^n.
\end{equation}
Here $n\geq1$, the differential inequality holds almost everywhere
in time, and the moment bound holds for every $t\geq0$. The
constant $C$ is absolute. See \cite[Appendix, Lemma~A.1 and
its proof]{ISG1999}. In the confinement argument, the moment order
is chosen proportional to $\log t$, so that the tail is small
enough to control the radial velocity. The factor $n$ in the
second estimate of \eqref{eq:classical-moments} then contributes
the logarithm in \eqref{eq:classical-bound}.

We use the moments
\begin{equation}\label{eq:moments}
 H_n(t)=\int_{\mathbb R^2}|\mathbf x|^{4n+2}
                  \omega(t,\mathbf x)\,\mathrm d\mathbf x,
 \qquad n\geq0,
\end{equation}
where the factor $|\mathbf x|^2$ gives $H_0(t)=J$ by
Lemma~\ref{lem:conservation}. The increment of four in the exponent
is chosen so that the kernel estimate below produces products of
the form $H_j(t)H_{n-1-j}(t)$.
After symmetrization, the classical first-moment
cancellation removes two terms from the interaction kernel.
We then factor the remaining kernel and estimate it by
\[
 2\sum_{j=0}^{n-1}
       |\mathbf x|^{4j+2}|\mathbf y|^{4(n-1-j)+2}.
\]
The coefficient is independent of $n$, including when the radii
$|\mathbf x|$ and $|\mathbf y|$ are close. This yields
\begin{equation}\label{eq:convolution-intro}
 |H_n'(t)|\leq\frac{2n+1}{\pi}
                 \sum_{j=0}^{n-1}H_j(t)H_{n-1-j}(t)
\end{equation}
for every integer $n\geq1$ and almost every $t\geq0$.

The intermediate products in \eqref{eq:convolution-intro} are
crucial. Indeed, for $n\geq2$ and $0\leq j\leq n-1$,
H\"older's inequality implies
\[
 H_j(t)\leq J^{1-j/(n-1)}H_{n-1}(t)^{j/(n-1)},\qquad
 H_j(t)H_{n-1-j}(t)\leq JH_{n-1}(t).
\]
Replacing each product by this upper bound would introduce an
additional factor $n$. We instead retain the convolution and,
for each integer $N\geq1$, apply \eqref{eq:convolution-intro}
to the finite polynomial
\[
 F_N(t,z)=\frac1J\sum_{n=0}^NH_n(t)z^n,
 \qquad t\geq0,\quad z\geq0.
\]
The coefficients of $F_N^2$ contain the convolution sums in
\eqref{eq:convolution-intro}, leading to the differential inequality
\eqref{eq:polynomial-inequality}. We choose a curve $z=z(t)>0$ and
a comparison function $A(t)>0$ so that, along this curve, the
$\partial_zF_N$ term cancels when $F_N=A$, while $A'$ matches the
remaining quadratic term; see \eqref{eq:comparison-odes}.
An elementary differential comparison then gives
$F_N(t,z(t))\leq A(t)$ independently of $N$. This controls
the growth of $H_n(t)$ as $n\to\infty$. At each fixed time, the
resulting tail estimate yields
\[
 R(t)^4\leq(R_0^4+\delta)
       \left(1+\frac{Jt}{\pi\delta}\right)^4,
 \qquad \delta>0.
\]
Choosing $\delta$ in terms of $R_0,J$ and $t$ proves \eqref{eq:main}.

The improvement therefore rests on the factorization of the
symmetrized kernel and on retaining the resulting moment
convolution. The conservation laws and the first-moment
cancellation are the same as in the earlier confinement arguments.
The moment identities are justified by compactly supported test
functions, without a preliminary bound on the support or velocity.

The rest of the paper is organized as follows.
Section~\ref{sec:moments} establishes the moment inequality. In
Section~\ref{sec:confinement}, we prove Theorem~\ref{thm:main}
and Corollary~\ref{cor:time-holder}.

\section{Estimates for high-order moments}\label{sec:moments}

\subsection{Weak formulation and moment identities}\label{subsec:transport-moments}

The following lemma establishes the
conservation laws and moment identities in the full $L^1$ class
\eqref{eq:solution-class}--\eqref{eq:weak-formulation}.

\begin{lemma}\label{lem:conservation}
Let $\omega$ be nonnegative and satisfy
\eqref{eq:solution-class}--\eqref{eq:weak-formulation}, with initial
data satisfying \eqref{eq:data}. Then, for every $t\geq0$,
\[
 \begin{aligned}
 \int_{\mathbb R^2}\omega(t,\mathbf x)\,\mathrm d\mathbf x&=\Gamma,\qquad
 \int_{\mathbb R^2}\mathbf x\omega(t,\mathbf x)\,\mathrm d\mathbf x=\Gamma\mathbf c,\qquad
 \int_{\mathbb R^2}|\mathbf x-\mathbf c|^2\omega(t,\mathbf x)\,\mathrm d\mathbf x=J.
 \end{aligned}
\]
For each integer $k\geq0$, the moment
\[
 M_k(t)=\int_{\mathbb R^2}|\mathbf x|^{2k}\omega(t,\mathbf x)\,\mathrm d\mathbf x
\]
is finite and locally absolutely continuous on $[0,\infty)$.
For $k\geq1$, set $f_k(\mathbf x)=|\mathbf x|^{2k}$ and define
$\mathcal H_{f_k}$ by the same formula as $\mathcal H_\varphi$.
Then, for every $t\geq0$,
\begin{equation}\label{eq:polynomial-identity}
 \begin{aligned}
 &M_k(t)-M_k(0)
=\int_0^t\iint_{\mathbb R^2\times\mathbb R^2}\mathcal H_{f_k}(\mathbf x,\mathbf y)
       \omega(s,\mathbf x)\omega(s,\mathbf y)
       \,\mathrm d\mathbf x\,\mathrm d\mathbf y\,\mathrm ds.
 \end{aligned}
\end{equation}
\end{lemma}

\begin{proof}
Fix $T>0$. Weak continuity in $L^1$ and the uniform boundedness
principle imply
\[
 G_T:=\sup_{0\leq t\leq T}\|\omega(t,\cdot)\|_{L^1}<\infty.
\]
Choose $\chi\in C_c^\infty(\mathbb R^2)$ such that
$0\leq\chi\leq1$, $\chi=1$ on $B_1(\mathbf0)$, and
$\chi=0$ outside $B_2(\mathbf0)$. For $R\geq1$, put
$\chi_R(\mathbf x)=\chi(\mathbf x/R)$. Since
\[
 \|D^2\chi_R\|_{L^\infty}=R^{-2}\|D^2\chi\|_{L^\infty},
\]
identity \eqref{eq:weak-formulation} implies
\[
 \begin{aligned}
 &\left|\int_{\mathbb R^2}\chi_R(\mathbf x)\omega(t,\mathbf x)\,\mathrm d\mathbf x
       -\int_{\mathbb R^2}\chi_R(\mathbf x)\omega_0(\mathbf x)\,\mathrm d\mathbf x\right|\leq\frac{TG_T^2}{4\pi R^2}\|D^2\chi\|_{L^\infty}.
 \end{aligned}
\]
Letting $R\to\infty$ and using dominated convergence on the left,
we obtain
\begin{equation}\label{eq:circulation}
 \int_{\mathbb R^2}\omega(t,\mathbf x)\,\mathrm d\mathbf x=\Gamma,
 \qquad 0\leq t\leq T.
\end{equation}

We next show inductively that $M_k$ is bounded on $[0,T]$ for every
$k$. The case $k=0$ follows from \eqref{eq:circulation}.
For $k\geq1$, set $\varphi_{k,R}=f_k\chi_R$.
Direct computation implies that
\[
 \begin{aligned}
 D^2\varphi_{k,R}
 &=\chi_R D^2f_k
   +R^{-1}\bigl(\nabla f_k\otimes\nabla\chi(\mathbf x/R)
              +\nabla\chi(\mathbf x/R)\otimes\nabla f_k\bigr)+R^{-2}f_kD^2\chi(\mathbf x/R).
 \end{aligned}
\]
Here $\otimes$ denotes the tensor product. The derivatives of $\chi$
in the last two terms vanish unless $R\leq|\mathbf x|\leq2R$.
Using $|\nabla f_k|=2k|\mathbf x|^{2k-1}$ and
$|D^2f_k|\leq2k(2k-1)|\mathbf x|^{2k-2}$, we obtain
\begin{equation}\label{eq:cutoff-hessian}
 |D^2\varphi_{k,R}(\mathbf x)|
 \leq C_k|\mathbf x|^{2k-2},
\end{equation}
where $C_k$ is independent of $R$. For $k=1$, the factor on the
right is interpreted as $1$. For $\mathbf x\neq\mathbf y$,
\[
 \nabla\varphi_{k,R}(\mathbf x)-\nabla\varphi_{k,R}(\mathbf y)
 =\int_0^1D^2\varphi_{k,R}((1-\theta)\mathbf y+\theta\mathbf x)
                      (\mathbf x-\mathbf y)\,\mathrm d\theta.
\]
Since $|(1-\theta)\mathbf y+\theta\mathbf x|
\leq\max\{|\mathbf x|,|\mathbf y|\}$, it follows that
\begin{equation}\label{eq:cutoff-kernel}
 |\mathcal H_{\varphi_{k,R}}(\mathbf x,\mathbf y)|
 \leq C_k\bigl(|\mathbf x|^{2k-2}+|\mathbf y|^{2k-2}\bigr).
\end{equation}
Consequently, \eqref{eq:weak-formulation} and the induction hypothesis
give, for every $0\leq t\leq T$,
\[
 0\leq\int_{\mathbb R^2}\varphi_{k,R}(\mathbf x)\omega(t,\mathbf x)\,\mathrm d\mathbf x
 \leq M_k(0)+2C_k\Gamma\int_0^tM_{k-1}(s)\,\mathrm ds.
\]
Fatou's lemma as $R\to\infty$ yields
\begin{equation}\label{eq:preliminary-moments}
 M_k(t)\leq M_k(0)+2C_k\Gamma\int_0^tM_{k-1}(s)\,\mathrm ds<\infty.
\end{equation}
The right-hand side is bounded on $[0,T]$, completing the induction.

We may now pass to the limit in \eqref{eq:weak-formulation} with
$\varphi=\varphi_{k,R}$. The left-hand side converges by the finiteness
of $M_k(t)$ and $M_k(0)$. The integrand on the right converges away
from the diagonal and, by \eqref{eq:cutoff-kernel}, is dominated by
an integrable function, since
\[
 \begin{aligned}
 &\int_0^T\iint_{\mathbb R^2\times\mathbb R^2}
 \bigl(|\mathbf x|^{2k-2}+|\mathbf y|^{2k-2}\bigr)
 \omega(s,\mathbf x)\omega(s,\mathbf y)
 \,\mathrm d\mathbf x\,\mathrm d\mathbf y\,\mathrm ds=2\Gamma\int_0^TM_{k-1}(s)\,\mathrm ds<\infty.
 \end{aligned}
\]
Thus \eqref{eq:polynomial-identity} holds for every $t\in[0,T]$
and proves local absolute continuity of $M_k$ for $k\geq1$.
For $k=1$, we have
\[
 \mathcal H_{f_1}(\mathbf x,\mathbf y)
 =\frac{2(\mathbf x-\mathbf y)\cdot(\mathbf x-\mathbf y)^\perp}
             {4\pi|\mathbf x-\mathbf y|^2}=0,
\]
so $M_1(t)=M_1(0)$.

For the first moments, use
$\psi_{i,R}(\mathbf x)=x_i\chi_R(\mathbf x)$, $i=1,2$.
If $\mathbf e_i$ is the $i$th coordinate vector, then
\[
 D^2\psi_{i,R}
 =R^{-1}\bigl(\mathbf e_i\otimes\nabla\chi(\mathbf x/R)
            +\nabla\chi(\mathbf x/R)\otimes\mathbf e_i\bigr)
   +R^{-2}x_iD^2\chi(\mathbf x/R).
\]
Thus $\|D^2\psi_{i,R}\|_{L^\infty}\leq C/R$ and
\[
 \begin{aligned}
 &\left|\int_{\mathbb R^2}\psi_{i,R}(\mathbf x)\omega(t,\mathbf x)\,\mathrm d\mathbf x
       -\int_{\mathbb R^2}\psi_{i,R}(\mathbf x)\omega_0(\mathbf x)\,\mathrm d\mathbf x\right|\leq\frac{CT\Gamma^2}{R}.
 \end{aligned}
\]
Cauchy--Schwarz gives
\[
 \int_{\mathbb R^2}|\mathbf x|\omega(t,\mathbf x)\,\mathrm d\mathbf x
 \leq\Gamma^{1/2}M_1(t)^{1/2}<\infty.
\]
Dominated convergence as $R\to\infty$ therefore yields
\[
 \int_{\mathbb R^2}\mathbf x\omega(t,\mathbf x)\,\mathrm d\mathbf x=\Gamma\mathbf c.
\]
Together with conservation of $M_1$ and $\Gamma$, this gives
\[
 \begin{aligned}
 &\int_{\mathbb R^2}|\mathbf x-\mathbf c|^2\omega(t,\mathbf x)\,\mathrm d\mathbf x=M_1(t)-2\mathbf c\cdot\int_{\mathbb R^2}\mathbf x\omega(t,\mathbf x)
                         \,\mathrm d\mathbf x+\Gamma|\mathbf c|^2=M_1(0)-\Gamma|\mathbf c|^2=J.
 \end{aligned}
\]
Since $T$ was arbitrary, all these conclusions hold on $[0,\infty)$.
\end{proof}

By translation invariance, we henceforth assume $\mathbf c=\mathbf0$.
The moments in Lemma~\ref{lem:conservation} remain finite in the
translated coordinates, as follows
from $|\mathbf x-\mathbf c|^{2k}
\leq2^{2k-1}(|\mathbf x|^{2k}+|\mathbf c|^{2k})$ for $k\geq1$.
The first-moment conservation in Lemma~\ref{lem:conservation} gives
\begin{equation}\label{eq:zero-first-moment}
 \int_{\mathbb R^2}\mathbf x\omega(t,\mathbf x)\,\mathrm d\mathbf x=\mathbf0,
 \qquad t\geq0.
\end{equation}
We use the moments $H_n$ in \eqref{eq:moments}.
By Lemma~\ref{lem:conservation} and the initial support condition,
\begin{equation}\label{eq:initial-moments}
 H_0(t)=J,\qquad H_n(0)\leq JR_0^{4n},\qquad n\geq1.
\end{equation}
Applying Lemma~\ref{lem:conservation} in these coordinates with
$k=2n+1$ shows that $H_n$ is locally absolutely continuous and that
\begin{equation}\label{eq:moment-derivative}
 H_n'(t)=\iint_{\mathbb R^2\times\mathbb R^2}\mathcal H_{f_{2n+1}}(\mathbf x,\mathbf y)
                   \omega(t,\mathbf x)\omega(t,\mathbf y)
                   \,\mathrm d\mathbf x\,\mathrm d\mathbf y
\end{equation}
for almost every $t\geq0$. 

\subsection{The interaction estimate}\label{subsec:interaction}

We first record the algebraic estimate used in the moment inequality.

\begin{lemma}\label{lem:kernel}
Let $n\geq1$ be an integer and let $\mathbf x,\mathbf y\in\mathbb R^2$
with $\mathbf x\neq\mathbf y$. Write
$[\mathbf x,\mathbf y]=x_1y_2-x_2y_1$ and define
\begin{equation}\label{eq:polynomial}
 P_n(\mathbf x,\mathbf y)
 =\sum_{j=0}^{n-1}|\mathbf x|^{4j}|\mathbf y|^{4n-4-4j}
\end{equation}
and
\begin{equation}\label{eq:kernel}
 \mathcal K_n(\mathbf x,\mathbf y)
 =\left\{\frac{|\mathbf x|^{4n}-|\mathbf y|^{4n}}
                    {|\mathbf x-\mathbf y|^2}
           -\bigl(|\mathbf x|^{4n-2}-|\mathbf y|^{4n-2}\bigr)\right\}
      [\mathbf x,\mathbf y].
\end{equation}
Then
\begin{equation}\label{eq:kernel-factorization}
 \begin{aligned}
 \mathcal K_n(\mathbf x,\mathbf y)
 &=P_n(\mathbf x,\mathbf y)\mathcal K_1(\mathbf x,\mathbf y)-\frac{|\mathbf x|^2|\mathbf y|^2}{|\mathbf x|^2+|\mathbf y|^2}
       \bigl(|\mathbf x|^{4n-4}-|\mathbf y|^{4n-4}\bigr)
       [\mathbf x,\mathbf y]
 \end{aligned}
\end{equation}
and
\begin{equation}\label{eq:kernel-bound}
 |\mathcal K_n(\mathbf x,\mathbf y)|
 \leq2|\mathbf x|^2|\mathbf y|^2P_n(\mathbf x,\mathbf y).
\end{equation}
\end{lemma}

\begin{proof}
Direct computation yields
\[
 \begin{aligned}
 |\mathbf x-\mathbf y|^2
 &=|\mathbf x|^2+|\mathbf y|^2-2\mathbf x\cdot\mathbf y,\qquad
 [\mathbf x,\mathbf y]^2
 =|\mathbf x|^2|\mathbf y|^2-(\mathbf x\cdot\mathbf y)^2.
 \end{aligned}
\]
For $n=1$, we have
\begin{equation}\label{K1 expression}
    \begin{aligned}
 \mathcal K_1(\mathbf x,\mathbf y)
 &=\bigl(|\mathbf x|^2-|\mathbf y|^2\bigr)
   \left(\frac{|\mathbf x|^2+|\mathbf y|^2}{|\mathbf x-\mathbf y|^2}-1\right)
   [\mathbf x,\mathbf y]=\frac{2\bigl(|\mathbf x|^2-|\mathbf y|^2\bigr)
             (\mathbf x\cdot\mathbf y)[\mathbf x,\mathbf y]}
            {|\mathbf x-\mathbf y|^2}.
 \end{aligned}
\end{equation}
 
Moreover,
\[
 \begin{aligned}
 &|\mathbf x|^2|\mathbf y|^2|\mathbf x-\mathbf y|^4
   -\bigl(|\mathbf x|^2-|\mathbf y|^2\bigr)^2[\mathbf x,\mathbf y]^2\\
 &\quad=|\mathbf x|^2|\mathbf y|^2
     \bigl(|\mathbf x|^2+|\mathbf y|^2-2\mathbf x\cdot\mathbf y\bigr)^2-\bigl(|\mathbf x|^2-|\mathbf y|^2\bigr)^2
     \bigl(|\mathbf x|^2|\mathbf y|^2-(\mathbf x\cdot\mathbf y)^2\bigr)\\
 &\quad=\bigl(|\mathbf x|^2+|\mathbf y|^2\bigr)^2
               (\mathbf x\cdot\mathbf y)^2-4|\mathbf x|^2|\mathbf y|^2
       \bigl(|\mathbf x|^2+|\mathbf y|^2\bigr)(\mathbf x\cdot\mathbf y)
       +4|\mathbf x|^4|\mathbf y|^4\\
 &\quad=\left(\bigl(|\mathbf x|^2+|\mathbf y|^2\bigr)
                     (\mathbf x\cdot\mathbf y)
                    -2|\mathbf x|^2|\mathbf y|^2\right)^2\\ &\quad\geq0.
 \end{aligned}
\]
Consequently,
\[
 \bigl||\mathbf x|^2-|\mathbf y|^2\bigr|\,
       |[\mathbf x,\mathbf y]|
 \leq|\mathbf x||\mathbf y||\mathbf x-\mathbf y|^2,
\]
and hence
\begin{equation}\label{eq:first-kernel-bound}
 |\mathcal K_1(\mathbf x,\mathbf y)|
 \leq2|\mathbf x||\mathbf y||\mathbf x\cdot\mathbf y|
 \leq2|\mathbf x|^2|\mathbf y|^2.
\end{equation}

By cancellation of successive terms in \eqref{eq:polynomial},
\begin{equation}\label{the first identity}
     \begin{aligned}
 &\bigl(|\mathbf x|^4-|\mathbf y|^4\bigr)P_n(\mathbf x,\mathbf y)=\sum_{j=0}^{n-1}
    \left(|\mathbf x|^{4j+4}|\mathbf y|^{4n-4-4j}
          -|\mathbf x|^{4j}|\mathbf y|^{4n-4j}\right)=|\mathbf x|^{4n}-|\mathbf y|^{4n}.
 \end{aligned}
\end{equation}

Furthermore, we obtain
\begin{equation}\label{the second identity}
    \begin{aligned}
 &\bigl(|\mathbf x|^2+|\mathbf y|^2\bigr)
   \left[\bigl(|\mathbf x|^2-|\mathbf y|^2\bigr)P_n(\mathbf x,\mathbf y)
              -\bigl(|\mathbf x|^{4n-2}-|\mathbf y|^{4n-2}\bigr)\right]\\
 &\quad=|\mathbf x|^{4n}-|\mathbf y|^{4n}
       -\bigl(|\mathbf x|^2+|\mathbf y|^2\bigr)
        \bigl(|\mathbf x|^{4n-2}-|\mathbf y|^{4n-2}\bigr)\\
 &\quad=-|\mathbf x|^2|\mathbf y|^2
         \bigl(|\mathbf x|^{4n-4}-|\mathbf y|^{4n-4}\bigr).
 \end{aligned}
\end{equation}
Note that \eqref{the first identity} cancels the terms with denominator
$|\mathbf x-\mathbf y|^2$ in
$\mathcal K_n(\mathbf x,\mathbf y)
 -P_n(\mathbf x,\mathbf y)\mathcal K_1(\mathbf x,\mathbf y)$, so that
\[
 \begin{aligned}
 &\mathcal K_n(\mathbf x,\mathbf y)
        -P_n(\mathbf x,\mathbf y)\mathcal K_1(\mathbf x,\mathbf y)=\big(\big(|\mathbf x|^2-|\mathbf y|^2\bigr)P_n(\mathbf x,\mathbf y)
              -\bigl(|\mathbf x|^{4n-2}-|\mathbf y|^{4n-2}\big)\big)
              [\mathbf x,\mathbf y].
 \end{aligned}
\]
Since $|\mathbf x|^2+|\mathbf y|^2>0$, using \eqref{the second identity} yields
\eqref{eq:kernel-factorization}. No division by
$|\mathbf x|^2-|\mathbf y|^2$ is involved, so the factorization also
holds when $|\mathbf x|=|\mathbf y|$.

It remains to prove \eqref{eq:kernel-bound}. If
$|\mathbf x||\mathbf y|=0$ or $|\mathbf x|=|\mathbf y|$, then
$\mathcal K_n(\mathbf x,\mathbf y)=0$. Otherwise,
$\mathcal K_n(\mathbf y,\mathbf x)=\mathcal K_n(\mathbf x,\mathbf y)$,
so exchanging $\mathbf x$ and $\mathbf y$ allows us to assume
$|\mathbf x|>|\mathbf y|>0$.
For $n=1$, the difference
$|\mathbf x|^{4n-4}-|\mathbf y|^{4n-4}$ is zero. For $n\geq2$,
the term $|\mathbf x|^{4n-4}$ occurs in $P_n(\mathbf x,\mathbf y)$,
and therefore
\[
 0\leq|\mathbf x|^{4n-4}-|\mathbf y|^{4n-4}
 \leq|\mathbf x|^{4n-4}\leq P_n(\mathbf x,\mathbf y).
\]
We also have
\[
 \frac{|[\mathbf x,\mathbf y]|}{|\mathbf x|^2+|\mathbf y|^2}
 \leq\frac{|\mathbf x||\mathbf y|}{|\mathbf x|^2+|\mathbf y|^2}
 \leq\frac12.
\]
The factorization can now be written as
\[
 \begin{aligned}
 \mathcal K_n(\mathbf x,\mathbf y)
 =[\mathbf x,\mathbf y]\Biggl(
 &\frac{2\bigl(|\mathbf x|^2-|\mathbf y|^2\bigr)
          (\mathbf x\cdot\mathbf y)P_n(\mathbf x,\mathbf y)}
         {|\mathbf x-\mathbf y|^2}-\frac{|\mathbf x|^2|\mathbf y|^2
          \bigl(|\mathbf x|^{4n-4}-|\mathbf y|^{4n-4}\bigr)}
         {|\mathbf x|^2+|\mathbf y|^2}\Biggr).
 \end{aligned}
\]
If $\mathbf x\cdot\mathbf y\geq0$, by \eqref{eq:first-kernel-bound}, we obtain
\[
 \begin{aligned}
 |\mathcal K_n(\mathbf x,\mathbf y)|
 &\leq\max\Biggl\{P_n(\mathbf x,\mathbf y)|\mathcal K_1(\mathbf x,\mathbf y)|,
       \frac{|\mathbf x|^2|\mathbf y|^2
              \bigl(|\mathbf x|^{4n-4}-|\mathbf y|^{4n-4}\bigr)
              |[\mathbf x,\mathbf y]|}
             {|\mathbf x|^2+|\mathbf y|^2}\Biggr\}\\
 &\leq\max\left\{2|\mathbf x|^2|\mathbf y|^2P_n(\mathbf x,\mathbf y),
                 \frac12|\mathbf x|^2|\mathbf y|^2P_n(\mathbf x,\mathbf y)\right\}\\
 &=2|\mathbf x|^2|\mathbf y|^2P_n(\mathbf x,\mathbf y).
 \end{aligned}
\]
If $\mathbf x\cdot\mathbf y<0$, then we have
\begin{equation*}
    |\mathbf x-\mathbf y|^2\geq|\mathbf x|^2+|\mathbf y|^2,
\end{equation*}
and
\begin{equation*}
    2|(\mathbf x\cdot\mathbf y)[\mathbf x,\mathbf y]|
 \leq(\mathbf x\cdot\mathbf y)^2+[\mathbf x,\mathbf y]^2
   =|\mathbf x|^2|\mathbf y|^2.
\end{equation*}
Thus, we obtain
\[
 \begin{aligned}
 |\mathcal K_n(\mathbf x,\mathbf y)|
 &\leq\frac{2\bigl(|\mathbf x|^2-|\mathbf y|^2\bigr)
                   |(\mathbf x\cdot\mathbf y)[\mathbf x,\mathbf y]|
                   P_n(\mathbf x,\mathbf y)}{|\mathbf x-\mathbf y|^2}+\frac{|\mathbf x|^2|\mathbf y|^2
                   \bigl(|\mathbf x|^{4n-4}-|\mathbf y|^{4n-4}\bigr)
                   |[\mathbf x,\mathbf y]|}
                  {|\mathbf x|^2+|\mathbf y|^2}\\
 &\leq\frac{|\mathbf x|^2-|\mathbf y|^2}{|\mathbf x|^2+|\mathbf y|^2}
           |\mathbf x|^2|\mathbf y|^2P_n(\mathbf x,\mathbf y)
        +\frac12|\mathbf x|^2|\mathbf y|^2P_n(\mathbf x,\mathbf y)\\
 &\leq\frac32|\mathbf x|^2|\mathbf y|^2P_n(\mathbf x,\mathbf y),
 \end{aligned}
\]
which proves \eqref{eq:kernel-bound} in all cases.
\end{proof}

\begin{proposition}\label{prop:moment-inequality}
For every integer $n\geq1$ and almost every $t\geq0$,
\begin{equation}\label{eq:convolution}
 |H_n'(t)|\leq\frac{2n+1}{\pi}
                   \sum_{j=0}^{n-1}H_j(t)H_{n-1-j}(t).
\end{equation}
\end{proposition}

\begin{proof}
We use the notation of Lemma~\ref{lem:kernel}.
The identity
$\mathbf x\cdot(\mathbf x-\mathbf y)^\perp=\mathbf y\cdot(\mathbf x-\mathbf y)^\perp=[\mathbf x,\mathbf y]$
implies that
\[
 \begin{aligned}
 &[\nabla f_{2n+1}(\mathbf x)-\nabla f_{2n+1}(\mathbf y)]
                    \cdot(\mathbf x-\mathbf y)^\perp=(4n+2)\bigl(|\mathbf x|^{4n}-|\mathbf y|^{4n}\bigr)
          [\mathbf x,\mathbf y].
 \end{aligned}
\]
Therefore, \eqref{eq:moment-derivative} can be rewritten as
\[
 \begin{aligned}
 H_n'(t)=\frac{4n+2}{4\pi}
 &\iint_{\mathbb R^2\times\mathbb R^2}
       \frac{|\mathbf x|^{4n}-|\mathbf y|^{4n}}
            {|\mathbf x-\mathbf y|^2}[\mathbf x,\mathbf y]\omega(t,\mathbf x)\omega(t,\mathbf y)
       \,\mathrm d\mathbf x\,\mathrm d\mathbf y.
 \end{aligned}
\]

By the first-moment conservation in Lemma~\ref{lem:conservation},
in the form \eqref{eq:zero-first-moment},
\[
 \begin{aligned}
 &\iint_{\mathbb R^2\times\mathbb R^2}
       |\mathbf x|^{4n-2}[\mathbf x,\mathbf y]
       \omega(t,\mathbf x)\omega(t,\mathbf y)
       \,\mathrm d\mathbf x\,\mathrm d\mathbf y=\left[
       \int_{\mathbb R^2}|\mathbf x|^{4n-2}\mathbf x\omega(t,\mathbf x)
                            \,\mathrm d\mathbf x,
       \int_{\mathbb R^2}\mathbf y\omega(t,\mathbf y)
                            \,\mathrm d\mathbf y
       \right]=0.
 \end{aligned}
\]
Similarly,
\[
 \begin{aligned}
 &\iint_{\mathbb R^2\times\mathbb R^2}
       |\mathbf y|^{4n-2}[\mathbf x,\mathbf y]
       \omega(t,\mathbf x)\omega(t,\mathbf y)
       \,\mathrm d\mathbf x\,\mathrm d\mathbf y=\left[
       \int_{\mathbb R^2}\mathbf x\omega(t,\mathbf x)
                            \,\mathrm d\mathbf x,
       \int_{\mathbb R^2}|\mathbf y|^{4n-2}\mathbf y\omega(t,\mathbf y)
                            \,\mathrm d\mathbf y
       \right]=0.
 \end{aligned}
\]
Subtracting these zero integrals, we obtain
\begin{equation}\label{eq:kernel-identity}
 \begin{aligned}
 H_n'(t)=\frac{4n+2}{4\pi}
 &\iint_{\mathbb R^2\times\mathbb R^2}
 \mathcal K_n(\mathbf x,\mathbf y)\omega(t,\mathbf x)\omega(t,\mathbf y)
 \,\mathrm d\mathbf x\,\mathrm d\mathbf y.
 \end{aligned}
\end{equation}

Nonnegativity of $\omega(t,\cdot)$ and Lemma~\ref{lem:kernel} now imply that
\[
 \begin{aligned}
 |H_n'(t)|\leq\frac{2n+1}{\pi}
 &\iint_{\mathbb R^2\times\mathbb R^2}
       |\mathbf x|^2|\mathbf y|^2P_n(\mathbf x,\mathbf y)\omega(t,\mathbf x)\omega(t,\mathbf y)
       \,\mathrm d\mathbf x\,\mathrm d\mathbf y.
 \end{aligned}
\]
Finally, Fubini's theorem and the definition of the moments yield
\[
 \begin{aligned}
 &\iint_{\mathbb R^2\times\mathbb R^2}
       |\mathbf x|^2|\mathbf y|^2P_n(\mathbf x,\mathbf y)
       \omega(t,\mathbf x)\omega(t,\mathbf y)
                         \,\mathrm d\mathbf x\,\mathrm d\mathbf y\\
 &\quad=\sum_{j=0}^{n-1}
 \left(\int_{\mathbb R^2}|\mathbf x|^{4j+2}\omega(t,\mathbf x)
                                     \,\mathrm d\mathbf x\right)
 \left(\int_{\mathbb R^2}|\mathbf y|^{4(n-1-j)+2}\omega(t,\mathbf y)
                                     \,\mathrm d\mathbf y\right)\\
 &\quad=\sum_{j=0}^{n-1}H_j(t)H_{n-1-j}(t),
 \end{aligned}
\]
which proves \eqref{eq:convolution}.
\end{proof}

\section{Confinement of the vorticity support}\label{sec:confinement}

In this section, we proceed to the proof of Theorem~\ref{thm:main}. Without loss of generality, we continue to assume that $\mathbf c=\mathbf0$.

\begin{proof}[Proof of Theorem~\ref{thm:main}]
For each integer $N\geq1$, define
\begin{equation}\label{eq:finite-polynomial}
 F_N(t,z)=\frac1J\sum_{n=0}^NH_n(t)z^n,
 \qquad t\geq0,\quad z\geq0.
\end{equation}
Its coefficients are nonnegative and, by Lemma~\ref{lem:conservation},
locally absolutely continuous in time. Write
\[
 F_N(t,z)^2=\sum_{m=0}^{2N}b_m(t)z^m.
\]
Then $b_m(t)\geq0$, and for $0\leq m\leq N-1$,
\[
 b_m(t)=\frac1{J^2}\sum_{j=0}^mH_j(t)H_{m-j}(t).
\]
By Proposition~\ref{prop:moment-inequality}, for almost every $t$
and every $z\geq0$,
\begin{equation}\label{eq:polynomial-inequality}
 \begin{aligned}
 \partial_tF_N(t,z)
 &\leq\frac J\pi\sum_{n=1}^N(2n+1)b_{n-1}(t)z^n\\
 &\leq\frac J\pi\sum_{m=0}^{2N}(2m+3)b_m(t)z^{m+1}\\
 &=\frac J\pi\left(3zF_N(t,z)^2
                +2z^2\partial_z\bigl(F_N(t,z)^2\bigr)\right)\\
 &=\frac J\pi\left(3zF_N(t,z)^2
                +4z^2F_N(t,z)\partial_zF_N(t,z)\right).
 \end{aligned}
\end{equation}

Fix $\delta>0$ and define
\begin{equation}\label{eq:comparison-functions}
 \begin{aligned}
 z(t):&=\frac1{R_0^4+\delta}
       \left(1+\frac{Jt}{\pi\delta}\right)^{-4},\qquad
 A(t):=\frac{R_0^4+\delta}{\delta}
       \left(1+\frac{Jt}{\pi\delta}\right)^3.
 \end{aligned}
\end{equation}
Both $A$ and $z$ are positive, and direct differentiation yields
\begin{equation}\label{eq:comparison-odes}
 \begin{aligned}
 z'(t)&=-\frac{4J}{\pi\delta(R_0^4+\delta)}
        \left(1+\frac{Jt}{\pi\delta}\right)^{-5}
       =-\frac{4J}{\pi}z(t)^2A(t),\\
 A'(t)&=\frac{3J(R_0^4+\delta)}{\pi\delta^2}
        \left(1+\frac{Jt}{\pi\delta}\right)^2
       =\frac{3J}{\pi}z(t)A(t)^2.
 \end{aligned}
\end{equation}
Moreover, from \eqref{eq:initial-moments}, we obtain
\[
 \begin{aligned}
 F_N(0,z(0))
 &\leq\sum_{n=0}^N\left(\frac{R_0^4}{R_0^4+\delta}\right)^n\leq\frac1{1-R_0^4/(R_0^4+\delta)}
 =\frac{R_0^4+\delta}{\delta}=A(0).
 \end{aligned}
\]
Let
\[
 W_N(t)=F_N(t,z(t))-A(t).
\]
Combining \eqref{eq:polynomial-inequality}
and \eqref{eq:comparison-odes}, we have
\[
 \begin{aligned}
 W_N'
 &=\partial_tF_N+z'\partial_zF_N-A'\\
 &\leq\frac J\pi\left[
     3z(F_N^2-A^2)+4z^2(F_N-A)\partial_zF_N\right]\\
 &=\frac J\pi\left[3z(F_N+A)+4z^2\partial_zF_N\right]W_N,
 \end{aligned}
\]
where $F_N$ and its partial derivatives are evaluated at
$(t,z(t))$. Set
\[
 g_N(t)=\frac J\pi\left[
  3z(t)\bigl(F_N(t,z(t))+A(t)\bigr)
  +4z(t)^2\partial_zF_N(t,z(t))\right].
\]
For each fixed $N$, all moments involved are bounded on compact
time intervals, so $g_N$ is locally integrable. Multiplying
$W_N'\leq g_NW_N$ by the integrating factor, we obtain
\[
 \frac{\mathrm d}{\mathrm dt}
 \left[\exp\left(-\int_0^tg_N(s)\,\mathrm ds\right)W_N(t)\right]
 =\exp\left(-\int_0^tg_N(s)\,\mathrm ds\right)
       (W_N'-g_NW_N)\leq0
\]
for almost every $t\geq0$. Since $W_N(0)\leq0$ and $W_N$ is locally
absolutely continuous, integration proves
\begin{equation}\label{eq:polynomial-bound}
 F_N(t,z(t))\leq A(t),
 \qquad N\geq1,\quad t\geq0.
\end{equation}
Taking $N=n$ and using nonnegativity of the coefficients gives
\begin{equation}\label{eq:moment-bound}
 H_n(t)\leq JA(t)z(t)^{-n},
 \qquad n\geq1,\quad t\geq0.
\end{equation}

For every $n$, the integrated estimate \eqref{eq:moment-bound}
holds for all $t\geq0$. We now fix $t$ and $\delta$, and only then
let $n\to\infty$. Choose $r>0$ such that $r^4z(t)>1$.
Then for every $n\geq1$, we obtain
\[
 \begin{aligned}
 0\leq\int_{\mathbb R^2\setminus\overline B_r(\mathbf0)}|\mathbf x|^2\omega(t,\mathbf x)
                         \,\mathrm d\mathbf x
 &\leq r^{-4n}H_n(t)\leq JA(t)\left(\frac1{r^4z(t)}\right)^n.
 \end{aligned}
\]
The last expression tends to zero as $n\to\infty$. Since $r>0$,
\[
 0\leq\int_{\mathbb R^2\setminus\overline B_r(\mathbf0)}\omega(t,\mathbf x)\,\mathrm d\mathbf x
 \leq\frac1{r^2}\int_{\mathbb R^2\setminus\overline B_r(\mathbf0)}|\mathbf x|^2
             \omega(t,\mathbf x)\,\mathrm d\mathbf x=0.
\]
Nonnegativity gives $\omega(t,\cdot)=0$ almost everywhere on
$\{|\mathbf x|>r\}$.
It follows that $R(t)\leq r$ for every $r>z(t)^{-1/4}$. Therefore
\begin{equation}\label{eq:parameter-bound}
 R(t)^4\leq z(t)^{-1}
 =(R_0^4+\delta)\left(1+\frac{Jt}{\pi\delta}\right)^4,
 \qquad t\geq0,\quad \delta>0.
\end{equation}

For $t>0$, choose
\[
 \begin{aligned}
 \delta
 &=\sqrt3\,R_0^2\sqrt{\frac{Jt}{\pi}}+\frac{3Jt}{\pi}=\sqrt{\frac{3Jt}{\pi}}
      \left(R_0^2+\sqrt{\frac{3Jt}{\pi}}\right).
 \end{aligned}
\]
Then we have
\begin{equation*}
    \sqrt{R_0^4+\delta}
 =\sqrt{R_0^4+\sqrt3\,R_0^2\sqrt{\frac{Jt}{\pi}}
                      +\frac{3Jt}{\pi}}\leq R_0^2+\sqrt{\frac{3Jt}{\pi}},
\end{equation*}
and 
\begin{equation*}
    \frac{Jt}{\pi\delta}
 =\frac{\sqrt{Jt/\pi}}
        {\sqrt3\left(R_0^2+\sqrt{3Jt/\pi}\right)}.
\end{equation*}
Taking square roots in \eqref{eq:parameter-bound} and using
$R_0^2+\sqrt{3Jt/\pi}\geq\sqrt{3Jt/\pi}$, we obtain
\[
 \begin{aligned}
 R(t)^2
 &\leq\left(R_0^2+\sqrt{\frac{3Jt}{\pi}}\right)
       \left(1+\frac{\sqrt{Jt/\pi}}
       {\sqrt3\left(R_0^2+\sqrt{3Jt/\pi}\right)}\right)^2\\
 &=R_0^2+\sqrt{\frac{3Jt}{\pi}}
       +\frac2{\sqrt3}\sqrt{\frac{Jt}{\pi}}
       +\frac{Jt}{3\pi\left(R_0^2+\sqrt{3Jt/\pi}\right)}\\
 &\leq R_0^2+\left(\sqrt3+\frac2{\sqrt3}
                       +\frac1{3\sqrt3}\right)\sqrt{\frac{Jt}{\pi}}\\
 &=R_0^2+\frac{16}{\sqrt{27\pi}}\sqrt{Jt}.
 \end{aligned}
\]
This proves \eqref{eq:main} for $t>0$.  Translation restores the original center $\mathbf c$.
Finally, $\operatorname{diam}(\operatorname{supp}\omega(t,\cdot))\leq2R(t)$
proves \eqref{eq:asymptotic}.
\end{proof}

\begin{proof}[Proof of Corollary~\ref{cor:time-holder}]
Without loss of generality, we still assume that $\mathbf{c}=\mathbf 0$.
It suffices to consider $0\leq s<t$. Set $\Delta=t-s$ and, for
$0\leq\sigma\leq\Delta$, define
\[
 H_n^+(\sigma)=H_n(s+\sigma),\qquad
 H_n^-(\sigma)=H_n(t-\sigma).
\]
By Lemma~\ref{lem:conservation}, these functions are nonnegative
and absolutely continuous, with
$H_0^+=H_0^-=J$. Theorem~\ref{thm:main} gives $R(s),R(t)<\infty$,
and hence
\[
 H_n^+(0)\leq J R(s)^{4n},\qquad
 H_n^-(0)\leq J R(t)^{4n},\qquad n\geq1.
\]
Moreover, we have
\[
 (H_n^+)'(\sigma)=H_n'(s+\sigma),\qquad
 (H_n^-)'(\sigma)=-H_n'(t-\sigma).
\]
The absolute-value estimate in \eqref{eq:convolution} therefore implies that
\[
 (H_n^\pm)'(\sigma)
 \leq\frac{2n+1}{\pi}\sum_{j=0}^{n-1}
 H_j^\pm(\sigma)H_{n-1-j}^\pm(\sigma)
\]
for every $n\geq1$ and almost every $\sigma\in[0,\Delta]$.

The finite-sum argument
\eqref{eq:finite-polynomial}--\eqref{eq:moment-bound} uses only these
properties and the initial moment bounds. Applying this argument to $H_n^+$
with $R_0$ replaced by $R(s)$, and to $H_n^-$ with $R_0$ replaced
by $R(t)$, followed by the same tail estimate as in
\eqref{eq:parameter-bound}, yields
\[
 \begin{aligned}
 R(t)^4&\leq(R(s)^4+\delta)
       \left(1+\frac{J\Delta}{\pi\delta}\right)^4,\qquad
 R(s)^4\leq(R(t)^4+\delta)
       \left(1+\frac{J\Delta}{\pi\delta}\right)^4
 \end{aligned},
\]
for $\delta>0$.
In the first inequality choose
$\delta=\sqrt3\,R(s)^2\sqrt{J\Delta/\pi}+3J\Delta/\pi$, and in the
second choose
$\delta=\sqrt3\,R(t)^2\sqrt{J\Delta/\pi}+3J\Delta/\pi$.
The calculation following \eqref{eq:parameter-bound}, with $R_0$
replaced by $R(s)$ and $R(t)$, respectively, gives
\[
 \begin{aligned}
 R(t)^2&\leq R(s)^2+\frac{16}{\sqrt{27\pi}}\sqrt{J\Delta},\qquad
 R(s)^2\leq R(t)^2+\frac{16}{\sqrt{27\pi}}\sqrt{J\Delta}.
 \end{aligned}
\]
Combining these inequalities proves \eqref{eq:time-holder}.
\end{proof}

\section*{Declarations}

\subsection*{Acknowledgement}
D. Cao was supported by the National Key R\&D Program of China (2023YFA1010001) and the National Natural Science Foundation of
China (12371212). J. Fan was supported by the National Key R\&D Program of
China (2022YFA1005602) and the National Natural Science Foundation of
China (12371212). G. Wang was supported by the National Natural Science
Foundation of China (12471101).
The authors acknowledge the use of OpenAI tools for language editing and manuscript preparation, including assistance with the exposition and checking notation. The authors take full responsible for the content of the manuscript, including the all mathematical statements and proofs.

\subsection*{Author contributions}
All authors contributed equally.

\subsection*{Conflict of interest}
The authors declare that they have no conflict of interest.

\subsection*{Data availability}
Data sharing is not applicable to this article because no datasets were
generated or analyzed.

\end{document}